\documentclass[14pt, a4paper]{article}
\usepackage{}
\usepackage{amsfonts}
\usepackage{mathbbold}
\usepackage{bbm}
\usepackage{mathrsfs}
\usepackage{latexsym}
\usepackage{graphicx}
\usepackage{amssymb}

\newtheorem{theorem}{Theorem}[section]
\newtheorem{lemma}[theorem]{Lemma}

\title{{\Large \bf  Least $Q$-eigenvalues of nonbipartite 2-connected graphs \thanks{Supported by NSFC
(No. 11771376, 11571252), Natural Science Foundation of Guangdong Province (2019A1515011031), Foundation of Lingnan Normal
University(ZL1923), ``333" Project of  Jiangsu (2016), NSFCU of Jiangsu (16KJB110011).}}}
\author{Guanglong Yu$^{a,b}$\thanks{Corresponding authors, E-mail addresses:
yglong01@163.com (G. Yu), rockzhang76@tzc.edu.cn (H. Zhang).} ~ Lin Sun$^{a}$~
 ~ Chao Yan$^{c,d}$ ~ Yarong Wu$^e$ ~  Hailiang Zhang$^f$$^{\dag}$ ~
\\ ~ \\
{\footnotesize $^a$Department of Mathematics, Lingnan Normal
University,  Zhanjiang, 524048, Guangdong, China}\\ {\footnotesize $^b$Department of Mathematics, Yancheng Teachers
University, Yancheng, 224002, Jiangsu, China}\\
{\footnotesize $^c$Meteorological and Oceanographic College, Natinonal University of Defense Technology, Nanjing, 211101, Jiangsu, China}\\
{\footnotesize $^d$Department of Mathematics, Pujiang Institute, Nanjing Tech University, Nanjing, 211101, Jiangsu, China}\\
{\footnotesize $^e$ SMU college of art and science, Shanghai maritime
University, Shanghai, 200135, China}\\
{\footnotesize $^f$Department of Mathematics, Taizhou University, Linhai, Zhejiang, 317000, China}}
\date{}

\begin{document}
\maketitle

\begin{abstract}
Among all simple nonbipartite 2-connected graphs and among all  nonbipartite $\theta$-graphs, the minimum least $Q$-eigenvalues are completely determined, respectively.


\end{abstract}
\baselineskip 21.6pt

\section{Introduction}

\ \ \ \  For a graph $G$ with vertex set $V(G) = \{v_{0}, v_{1}, \ldots, v_{n-1}\}$ and edge set $E(G)$, denote by $A(G)$ (or $A$ for short) the adjacency matrix and $D(G)$ (or $D$ for short) the diagonal matrix of degrees, where $D(G)=
\mathrm{diag}(deg_{G}(v_{0})$, $deg_{G}(v_{1})$,
$\ldots$, $deg_{G}(v_{n-1}))$ with $deg_{G}(v_{i})$ (or $deg(v_{i})$ for short) denoting the degree of vertex $v_{i}$. $Q(G)= D(G) + A(G)$ is called the $signless$ $Laplacian$ matrix. Note that a network can be looked as a graph. Thus we can study the property of a network by study a graph including the algebraic or combinatoric optimal properties. Note that $Q(G)$ is symmetric, and positive semi-definite because $X^{T}Q(G)X=\sum_{v_{i}v_{j}\in E(G)} (x(v_{i})+x(v_{j}))^{2}\geq 0$ for any $X=(x(v_{0})$, $x(v_{1})$, $\ldots$, $x(v_{n-1}))^{T}\in R^{n}$ where $x(v_{i})$ corresponds to vertex $v_{i}$. Then the eigenvalues of $Q(G)$ are all nonnegative reals. The largest eigenvalue of $Q(G)$, denoted by $\rho(G)$, is called the $Q$-spectral radius of $G$; the least eigenvalue of $Q(G)$, denoted by $q(G)$, is called the least $Q$-eigenvalue of $G$. From spectral graph theory, it is known that for a system (multivariate function) $f_{G}(X)=\frac{X^{T}Q(G)X}{X^{T}X}=\sum_{v_{i}v_{j}\in E(G)} (x(v_{i})+x(v_{j}))^{2}/\sum_{i=0}^{n-1}x^{2}(v_{i})$ based on a graph $G$, under the condition $x(v_{i})\in R$ for $0\leq i\leq n-1$ and $X\neq \mathbf{0}^{T}$ (where $\mathbf{0}^{T}$ is the zero vector with all entries being 0), $\max f_{G}(X)$ is the $Q$-spectral radius of $Q(G)$; $\min f_{G}(X)$ is the least $Q$-eigenvalue of $Q(G)$. Also from spectral graph theory, it is known that a vector $X\in R^{n}$ satisfies that $f_{G}(X)\geq \rho(G)$ if and only if $X$ is an eigenvector of $G$ corresponding to $\rho(G)$; a vector $X\in R^{n}$ ($X\neq \mathbf{0}^{T}$) satisfies that $f_{G}(X)\leq q(G)$ if and only if $X$ is an eigenvector of $G$ corresponding to $q(G)$. The eigenvalues of a graph (spectral radius and the least eigenvalue in particular) are always used to study the properties of the communication in this graph \cite{DCSS}. Moreover, the least $Q$-eigenvalue of a graph was used to study the impulsive cluster anticonsensus problem of discrete
multiagent linear dynamic systems \cite{ZJ1, ZJ2}. The least $Q$-eigenvalue of a graph is also looked as a measure to discriminate the bipartiteness of this graph because for a connected graph, its least $Q$-eigenvalue is zero if and only if it is bipartite \cite{D.P.S}. These make the study of least $Q$-eigenvalue of a graph as a nice topic in spectral research for graphs.

Given a graph $G$, the cardinality $\|V(G)\| =n$ is always called the $order$, $\|E(G)\| =m$ is always called the $size$ (in this paper, we denote $\|S\|$ the cardinality of a set $S$).
Throughout this paper, the graphs considered are connected, simple (no loops and no multiple edges) and
undirected.

There are a lot of results about the $Q$-spectral radius but much fewer about the least $Q$-eigenvalue of a graph. Note again that for a connected graph, its least $Q$-eigenvalue is zero if and only if it is bipartite. Thus the research on the least $Q$-eigenvalue focus mainly on the nonbipartite graphs. Some nice results about the least $Q$-eigenvalue of a graph have been shown in \cite{CCRS}, \cite{D.P.S}-\cite{ZJ2} and the references therein. In \cite{GZY}, the authors investigated the least $Q$-eigenvalue among all the nonbipartite Hamiltonian graphs of given order and determined the least $Q$-eigenvalue. With the motivation to investigate the least $Q$-eigenvalue about the more general graphs, we consider the nonbipartite 2-connected graphs.

Now, we recall some notions and notations of a graph. For a graph $G$, denote by $N_{G}(u)$ the neighbor set of vertex $u$. For two different vertices  $u, v$ in $G$, we denote by $dist_{G}(u,v)$ the distance between $u$ and $v$ where the distance is the length of the shortest path from $u$ to $v$. We denote by $L(P)$, $L(C)$ the length of a path $P$ and a cycle $C$ respectively. A cycle with odd (even) length is called an $odd$ ($even$) cycle. A graph is called $complete$ if its any two different vertices are adjacent; otherwise, it is called $uncomplete$. Given a connected uncomplete graph $G$, if $G$ has a vertex subset $S\subset V(G)$ that $G-S$ is not connected where $G-S$ is the graph obtained from $G$ by deleting all the vertices in $S$ and all the edges incident with the vertices in $S$, then $S$ is called a $vertex$ $cut$ of $G$; the cardinality $\|S\|$ is called the capacity of vertex cut $S$. The smallest capacity among all the vertex cuts of $G$, denoted by $c(G)$, is called the $connectivity$ of $G$. If $G$ is a complete graph of order $n\geq 2$, we define $c(G)=n-1$. A graph with order $n\geq k+1$ and connectivity $k$ is called a $k$-$connected$ graph. For a connected nontrivial graph $G$ and a vertex $v$ in $G$, if $G-v$ is disconnected where $G-v$ is obtained from $G$ by deleting vertex $v$ and all the edges incident with $v$, then $v$ is called a $cut$ $vertex$. Obviously, a connected graph of order $n\geq 3$ is 1-connected if and only if it has a cut vertex. In a graph, two paths $P_{1}$ and $P_{2}$ from vertex $u$ to $v$ are called $inner$ $disjoint$ if $V(P_{1})\cap V(P_{2})=\{u, v\}$ (that is, no inner vertex in common). The $local$ $connectivity$ between two distinct vertices $u$ and $v$, denoted by $p(u, v)$, is the maximum number
of pairwise inner disjoint paths from $u$ to $v$. The famous Menger's Theorem (see \cite{JBUM1} and \cite{JBUM} for example) tells us that in a nontrivial connected graph $G$, $p(u, v)\geq c(G)$ for any pair of distinct vertices
$u$ and $v$ in $G$. Hence, in \cite{JBUM}, a $k$-$connected$ graph $G$ is also defined to be the graph in which $p(u, v)\geq k$ for its any two distinct vertices
$u$ and $v$. A $\theta$-graph is a 2-connected graph which consists of three pairwise inner disjoint paths with common initial and terminal vertices.

For a graph $G$, let $G +uv$ denote the graph
obtained from $G$ by adding a new edge $uv\notin E(G)$ between two nonadjacent vertices $u, v$ in $G$; let $G-uv$ denote the graph
obtained from $G$ by deleting an edge $uv\in E(G)$; for another graph $K$ with $\|E(K)\|\geq 1$, $E(K)\nsubseteq  E(G)$, let $G +K$ denote the graph
obtained from $G$ and $K$ with new vertex set $V(G+K)=V(G)\cup V(K)$ and new edge set $E(G+K)=E(G)\cup E(K)$, where $V(K)\cap  V(G)\neq \emptyset$ and $E(K)\cap  E(G)\neq \emptyset$ possibly. In this paper, we let $C_{n} = v_{0}v_{ 1} v_{ 2}\cdots v_{n-1}v_{0}$ be the cycle of order $n$, and let $H(i_{ 1}, \ldots, i_{ k})= C_{ n} +v_{ i_{ 1}} v_{ n-i_{1}} +\cdots+v_{ i_{ k}} v_{ n-i_{ k}} $ where $n \geq 5$ is odd, $1 \leq k \leq \frac{n -3}{2}$
and $1 \leq i_{ 1} \leq \cdots \leq i_{ k} \leq \frac{n -3}{2}$ (for example, see $H(1, 2, \ldots, \frac{n -3}{2})$ in Fig. 1.1).

\

\setlength{\unitlength}{0.6pt}
\begin{center}
\begin{picture}(719,135)
\put(25,79){\circle*{4}}
\put(85,47){\circle*{4}}
\qbezier(25,79)(55,63)(85,47)
\put(85,112){\circle*{4}}
\qbezier(25,79)(55,96)(85,112)
\put(142,112){\circle*{4}}
\qbezier(85,112)(113,112)(142,112)
\put(142,47){\circle*{4}}
\qbezier(85,47)(113,47)(142,47)
\qbezier(142,112)(142,80)(142,47)
\qbezier(85,112)(85,80)(85,47)
\put(180,108){\circle*{4}}
\put(196,108){\circle*{4}}
\put(211,108){\circle*{4}}
\put(181,51){\circle*{4}}
\put(197,51){\circle*{4}}
\put(212,51){\circle*{4}}
\put(244,112){\circle*{4}}
\put(244,50){\circle*{4}}
\qbezier(244,112)(244,81)(244,50)
\put(296,112){\circle*{4}}
\qbezier(244,112)(270,112)(296,112)
\put(296,50){\circle*{4}}
\qbezier(244,50)(270,50)(296,50)
\qbezier(296,112)(296,81)(296,50)
\put(6,78){$v_{0}$}
\put(80,33){$v_{1}$}
\put(138,33){$v_{2}$}
\put(286,35){$v_{\frac{n-1}{2}}$}
\put(288,123){$v_{\frac{n+1}{2}}$}
\put(130,120){$v_{n-2}$}
\put(72,120){$v_{n-1}$}
\put(78,-9){Fig. 1.1. $H(1, 2, \ldots, \frac{n -3}{2})$}
\put(226,35){$v_{\frac{n-3}{2}}$}
\put(556,115){\circle*{4}}
\put(556,35){\circle*{4}}
\qbezier(556,115)(556,75)(556,35)
\put(477,75){\circle*{4}}
\qbezier(556,115)(477,115)(477,75)
\qbezier(477,75)(477,33)(556,35)
\qbezier(556,115)(718,119)(717,75)
\qbezier(556,35)(719,31)(717,75)
\put(556,74){\circle*{4}}
\put(458,73){$v_{0}$}
\put(563,71){$v_{1}$}
\put(549,125){$v_{2}$}
\put(546,21){$v_{n-1}$}
\put(611,115){\circle*{4}}
\put(607,125){$v_{3}$}
\put(539,-8){Fig. 1.2. $\Theta$}
\end{picture}
\end{center}

Let $n\geq 4$ be an positive even integer, $\mathscr{C}=v_{1}v_{2}\cdots v_{n-1}v_{1}$ be an odd cycle. Let graph $\Theta(j,k)=\mathscr{C}+v_{0}v_{j}+v_{0}v_{k}$ where $j$, $k$ are positive integers that $1\leq j< k\leq n-1$, $\Theta=\Theta(2,n-1)$ (see Fig. 1.2). Denote by $P_{1}(\Theta(j,k))=v_{j}v_{j+1}\cdots v_{k-1}v_{k}$, $P_{2}(\Theta(j,k))=\mathscr{C}-\{v_{j+1}$, $\ldots$, $v_{k-1}\}=v_{j}v_{\eta_{1}}v_{\eta_{2}}\cdots v_{\eta_{z}}v_{k}$ if $k-j< n-2$ where $z=n-k+j-2$; $P_{2}(\Theta(j,k))=\mathscr{C}-\{v_{j+1}$, $\ldots$, $v_{k-1}\}=v_{j}v_{k}$ if $k-j= n-2$ (in fact, $j=1$ and $k=n-1$ now).

In this paper, for determining the minimum least Q-eigenvalues among all the simple nonbipartite 2-connected graphs and among all nonbipartite $\theta$-graphs, we explore some new results on the structural characteristics, on the characteristics of the eigenvector for a nonbipartite 2-connected graph, on the relation between the eigenvector and the structure of a graph, and represent some new results on the influence of the least $Q$-eigenvalue under some structural perturbations. Using these tool results, we determine  the minimum least $Q$-eigenvalues among all simple nonbipartite 2-connected  graphs and among all nonbipartite $\theta$-graphs as the following two theorems.

\begin{theorem}\label{th01.01} 
Let $G$ be a nonbipartite 2-connected graph of order $n\geq 3$.

$\mathrm{(i)}$ If $n$ is odd, then $q(C_{n})\leq q(G)$. Moreover,
if the equality holds, then $G \cong C_{ n}$ or $G \cong H(i_{ 1}, \ldots, i_{ k})$.

$\mathrm{(ii)}$ If $n$ is even, then $q(\Theta)\leq q(G)$ with equality if and only if $G\cong \Theta$.
\end{theorem}

\begin{theorem}\label{th01.02} 
Let $G$ be a nonbipartite $\theta$-graph of order $n\geq 4$.

$\mathrm{(i)}$ If $n$ is odd, then $q(C_{n})\leq q(G)$. Moreover,
if the equality holds, then $G \cong C_{ n}$ or $G \cong H(i)$ for some $1\leq i\leq (n -3)/2$.

$\mathrm{(ii)}$ If $n$ is even, then $q(\Theta)\leq q(G)$ with equality if and only if $G\cong \Theta$.
\end{theorem}

\section{Preliminary}

\ \ \ \ \ In this section, three working lemmas in this paper are introduced.

\begin{lemma}{\bf \cite{KDAS}}\label{le02,02} 
Let $G$ be a connected graph of order $n$. Then
$q<\delta$, where $\delta$ is the minimal vertex degree of $G$.
\end{lemma}

\begin{lemma}{\bf \cite{D.R.S}}\label{le02,03} 
Let $G$ be a connected graph of order $n$. Then
$q(G-e)\leq q(G)$.
\end{lemma}

\begin{lemma}{\bf \cite{RZSG}}\label{le02,04} 
Let $n$ be an odd positive integer, $G$ be a nonbipartite Hamiltonian graph of order $n$.
Then
$q(G)\geq q(C_{n})$. Moreover,
if the equality holds, then $G \cong C_{ n}$ or $G \cong H (i_{ 1}, \ldots, i_{ k})$.
\end{lemma}

\section{Main results}

\begin{lemma}\label{le03.01} 
Let $G$ be a $2$-connected graph, $C$ be a cycle in $G$, $e=uv$ be an edge in $C$, $\xi$ be a vertex not in $C$. Then
 in $G$, there are two paths $P_{1}$ which is from $\xi$ to $u$ and $P_{2}$ which is from $\xi$ to $v$ that $V(P_{1})\cap V(P_{2})= \{\xi\}$.

\end{lemma}

\begin{proof}
Let $P=uwv$ and $G^{'}=G-e+P$ where $w$ is a new vertex. Note that there is no cut vertex in $G^{'}$. Thus we get that $G^{'}$ is $2$-connected.

By Menger's Theorem mentioned in Section 1, it follows that in $G^{'}$, there are two inner disjoint paths $P^{'}_{1}$ and $P^{'}_{2}$ from $\xi$ to $w$, where $u$ is in $P^{'}_{1}$, $v$ is in $P^{'}_{2}$. Then the lemma follows from letting $P_{1}=P^{'}_{1}-w$, $P_{2}=P^{'}_{2}-w$.
\ \ \ \ \ $\Box$
\end{proof}

\begin{lemma}\label{cl03.02} 
Let $G$ be a $2$-connected graph, $C$ be a cycle in $G$, $\xi$ be a vertex not in $C$. Then there are two different paths $P_{1}$ and $P_{2}$ from $\xi$ to $C$ that $\|V(P_{1})\cap V(C)\|=1$, $\|V(P_{2})\cap V(C)\|=1$, $V(P_{1})\cap V(P_{2})= \{\xi\}$.
\end{lemma}

\begin{proof}
By Lemma \ref{le03.01}, for an edge $e=uv$ in $C$, there are two paths $P^{'}_{1}$ which is from $\xi$ to $u$ and $P^{'}_{2}$ which is from $\xi$ to $v$ that $V(P^{'}_{1})\cap V(P^{'}_{2})= \{\xi\}$. From $\xi$ to $u$ along $P^{'}_{1}$, denote by $w$ the first common vertex of $C$ and $P^{'}_{1}$, and denote by $P_{1}$ the path from $\xi$ to $w$ along $P^{'}_{1}$ ($P^{'}_{1}=P_{1}$ possible, $V(P_{1})\cap V(C)=\{w\}$). Similarly, we get a path $P_{2}$ from $\xi$ to cycle $C$ along $P^{'}_{2}$ that $\|V(P_{2})\cap V(C)\|=1$. Then the lemma follows. \ \ \ \ \ $\Box$
\end{proof}

\begin{lemma}\label{le03.03} 
Let $G$ be a nonbipartite $2$-connected graph with $V(G) = \{v_{0}, v_{1}, \ldots\}$, $X=(x(v_{0})$, $x(v_{1})$, $\ldots)^{T}$ be
an eigenvector of $G$ corresponding to $q(G)$ where $x(v_{i})\in R$ corresponds to vertex $v_{i}$, and $|x(v_{\mu})|=\max\{|x(v_{i})|\mid v_{i}\in V(G)\}$. Then $v_{\mu}$ must be in an odd cycle.
\end{lemma}

\begin{proof}
Suppose $C$ is an odd cycle in $G$. The lemma holds naturally if $v_{\mu}$ is in $C$. Suppose $v_{\mu}$ is not in $C$. Note that $G$ is $2$-connected. By Lemma \ref{cl03.02}, then there are two different paths $P_{1}$ and $P_{2}$ from $v_{\mu}$ to $C$ that $\|V(P_{1})\cap V(C)\|=1$, $\|V(P_{2})\cap V(C)\|=1$, $V(P_{1})\cap V(P_{2})= \{v_{\mu}\}$. Suppose $V(P_{1})\cap V(C)=\{v_{1}\}$, $V(P_{2})\cap V(C)=\{v_{2}\}$. Then by $v_{1}$ and $v_{2}$, $C$ is parted into two paths $W_{1}$, $W_{2}$ that $C= W_{1}\cup W_{2}$. Assume that $L(W_{1})$ is odd. Then $L(W_{2})$ is even. If $L(P_{1})+L(P_{2})$ is even, then $P_{1}\cup P_{2}\cup W_{1}$ is an odd cycle; if $L(P_{1})+L(P_{2})$ is odd, then $P_{1}\cup P_{2}\cup W_{2}$ is an odd cycle. This means that $v_{\mu}$ must be in an odd cycle. \ \ \ \ \ $\Box$
\end{proof}

\begin{theorem}\label{le03.04} 
Let $n\geq 3$ be a positive odd integer, $G$ be a nonbipartite $2$-connected graph of order $n$. Then $q(C_{n})\leq q(G)$. Moreover,
if the equality holds, then $G \cong C_{ n}$ or $G \cong H_{1} (i_{ 1}, \ldots, i_{ k})$.
\end{theorem}

\begin{proof}
Suppose $V(G) = \{v_{0}, v_{1}, \ldots, v_{n-1}\}$, $X=(x(v_{0})$, $x(v_{1})$, $\ldots$, $x(v_{n-1}))^{T}$ be
an eigenvector of $G$ corresponding to $q(G)$,
and $|x(v_{\mu})|=\max\{|x(v_{i})|\,|\, v_{i}\in V(G)\}$. Note that $X$ is not a zero vector. It follows that $x(v_{\mu})\neq 0$. By Lemma \ref{le03.03}, we know that $v_{\mu}$ is in an odd cycle $C$.

If $\|V(C)\|=n$, then the lemma follows from Lemma \ref{le02,04} immediately. Next, we suppose $\|V(C)\|<n$.

Note that $n$ is odd. It follows that $\|V(G)\setminus V(C)\|\geq 2$ is even. Without loss of generality, suppose $N_{C}(v_{\mu})=\{v_{0}, v_{1}\}$,
and suppose that $V(G)\setminus V(C)=\{v_{i_{1}}$, $v_{i_{2}}$, $\ldots$, $v_{i_{k}}\}$ where $k\geq 2$ is even. Now we let $P=v_{\mu}v_{i_{1}} v_{i_{2}}\cdots v_{i_{k}}v_{1}$, $\mathbb{C}=C-v_{\mu}v_{1}+P$, and $Y$ be a vector satisfying that
$$\left \{\begin{array}{ll}
 y(v_{j})=x(v_{j}),\ & \ v_{j}\in V(C);
\\ y(v_{i_{j}})=(-1)^{j}x(v_{\mu}),\ & \ j=1,2, \ldots, k. \end{array}\right.$$

Note that $Y^{T}Q(\mathbb{C})Y=\sum_{v_{\alpha}v_{\beta}\in E(C)} (x(v_{\alpha})+x(v_{\beta}))^{2}$, $Y^{T}Y\geq X^{T}X$. Then
$q(\mathbb{C})\leq\frac{Y^{T}Q(\mathbb{C})Y}{Y^{T}Y}\leq \frac{X^{T}Q(G)X}{X^{T}X}=q(G)$.

We claim that $q(\mathbb{C})<q(G)$. Otherwise, assume that $q(\mathbb{C})=q(G)$. Then $Y$ is an eigenvector of $\mathbb{C}$ corresponding to $q(\mathbb{C})$. Thus $q(\mathbb{C})y(v_{i_{1}})=2y(v_{i_{1}})+y(v_{\mu})+y(v_{i_{2}})$. Note $y(v_{i_{1}})=-x(v_{\mu})=-y(v_{i_{2}})\neq0$. Then $q(\mathbb{C})=0$, which contradicts $q(\mathbb{C})>0$ because $\mathbb{C}$ is nonbipartite. Consequently, our claim holds.

Note that $\mathbb{C}\cong C_{n}$. From the above discussion, we get that if $\|V(C)\|<n$, then $q(C_{n})<q(G)$. This implies that if $q(G)=q(C_{n})$, then $\|V(C)\|=n$, and then $G$ is Hamiltonian.
Combined Lemma \ref{le02,04}, the lemma follows. \ \ \ \ \ $\Box$
\end{proof}

Next, we consider the least $Q$-eigenvalues of the nonbipartite $2$-connected graphs of order $n$ for the case that $n\geq 4$ is even. Denote by $\mathcal {H}=\{G\mid G$ is a nonbipartite $2$-connected graph of order $n$ where $n\geq 4$ is an positive even integer$\}$.

\begin{lemma}\label{le03.05} 
Let $G\in \mathcal {H}$, $V(G) = \{v_{0}, v_{1}, \ldots, v_{n-1}\}$, $X=(x(v_{0})$, $x(v_{1})$, $\ldots$, $x(v_{n-1}))^{T}\in R^{n}$ be
an eigenvector of $G$ corresponding to $q(G)$,
and $|x(v_{\mu})|=\max\{|x(v_{i})|\,|\, v_{i}\in V(G)\}$, $v_{\mu}$ be in an odd cycle $C$. If $n-L(C)\geq 3$, then $q(\Theta)< q(G) $.
\end{lemma}

\begin{proof}
Suppose $V(G)\setminus V(C)=\{v_{i_{1}}$, $v_{i_{2}}$, $\ldots$, $v_{i_{k}}\}$ where $k\geq 3$ is odd, and suppose that $N_{C}(v_{\mu})=\{v_{j_{1}}, v_{j_{2}}\}$. Now we let $P_{1}=v_{\mu}v_{i_{1}}v_{i_{2}}\cdots v_{i_{k-1}}v_{j_{1}}$, $P_{2}=v_{\mu}v_{i_{k}}v_{i_{2}}$, $G^{'}=C-v_{\mu}v_{j_{1}}+P_{1}+P_{2}$, and $Y$ be a vector satisfying that
$$\left \{\begin{array}{ll}
 y(v_{t})=x(v_{t}),\ & \ v_{t}\in V(C);
\\ y(v_{i_{t}})=(-1)^{t}x(v_{\mu}),\ & \ t=1,2, \ldots, k. \end{array}\right.$$
Thus $y(v_{i_{1}})=-x(v_{\mu})=y(v_{i_{k}})$. Note that $Y^{T}Y\geq X^{T}X$, $Y^{T}Q(G^{'})Y=\sum_{v_{\alpha}v_{\beta}\in E(C)}(x(v_{\alpha})+x(v_{\beta}))^{2}\leq X^{T}Q(G)X$. Then $q(G^{'})\leq\frac{Y^{T}Q(G^{'})Y}{Y^{T}Y}\leq \frac{X^{T}Q(G)X}{X^{T}X}=q(G)$.

As proved for $q(\mathbb{C})<q(G)$ in the proof of Lemma \ref{le03.04} (by considering $q(G^{'})y(v_{i_{1}})$), we get that $q(G^{'})<q(G)$. Note that $G^{'}\cong \Theta$. Then the lemma follows. \ \ \ \ \ $\Box$
\end{proof}

\begin{lemma}\label{le03.06} 
If in an eigenvector $X=(x(v_{0})$, $x(v_{1})$, $\ldots$, $x(v_{n-1}))^{T}\in R^{n}$ of $\Theta(j,k)$ corresponding to $q(\Theta(j,k))$, there is a $x(v_{i})$ for $j\leq i<\frac{k+j}{2}$ that $x(v_{i})\neq x(v_{k+j-i})$, or there is a $x(v_{\eta_{i}})$ for $1\leq i<\frac{z+1}{2}$ that $x(v_{\eta_{i}})\neq x(v_{\eta_{z-i+1}})$, then there is an eigenvector $Y=(y(v_{0})$, $y(v_{1})$, $\ldots$, $y(v_{n-1}))^{T}\in R^{n}$ corresponding to $q(\Theta(j,k))$ that

(1) $y(v_{0})=0$;

(2) $y(v_{\frac{k+j}{2}})=0$ and $y(v_{\frac{k+j}{2}-1})\neq0$ if $k+j$ is even;

(3) $y(v_{\eta_{\frac{1+z}{2}}})=0$ and $y(v_{\eta_{\frac{1+z}{2}-1}})\neq0$ if $z\geq 3$ is odd; $y(v_{\eta_{1}})=0$ and $y(v_{j})\neq0$ if $z=1$;

(4) $y(v_{i})=-y(v_{k+j-i})$ for $j\leq i<\frac{k+j}{2}$;

(5) $y(v_{\eta_{i}})=- y(v_{\eta_{z-i+1}})$ for $1\leq i<\frac{z+1}{2}$ if $z\geq 2$.
\end{lemma}

\begin{proof}
Note that $\mathscr{C}=P_{1}(\Theta(j,k))\cup P_{2}(\Theta(j,k))$ in $\Theta(j,k)$, and note that $\mathscr{C}$ is an odd cycle. Then one of $\|V(P_{1}(\Theta(j,k)))\|$ and $\|V(P_{2}(\Theta(j,k)))\|$ is odd. Without loss of generality, assume that $\|V(P_{1}(\Theta(j,k)))\|$ is odd. Then $\|V(P_{2}(\Theta(j,k)))\|$, $k-j$, $k+j$ and $z$ (if $z\geq 2$) are all even.

For proving this lemma, we employ a vector $F=(f(v_{0})$, $f(v_{1})$, $\ldots$, $f(v_{n-1}))^{T}$ satisfying that $$\left \{\begin{array}{ll}
 f(v_{i})=-x(v_{k+j-i}), \ & j\leq i\leq \frac{k+j}{2}-1;\\
 \\
 f(v_{i})=-x(v_{j+k-i}), \ & \frac{k+j}{2}+1\leq i\leq k;\\
 \\
 f(v_{\frac{k+j}{2}})=-x(v_{\frac{k+j}{2}});\ & \\
 \\
 f(v_{0})=-x(v_{0});\ & \\
\\ f(v_{\eta_{i}})=- x(v_{\eta_{z-i+1}}),\ & \ 1\leq i\leq\frac{z}{2}\ if\ z\geq2;\\
\\ f(v_{\eta_{i}})=- x(v_{\eta_{1+z-i}}),\ & \ \frac{z}{2}+1\leq i\leq z\ if\ z\geq2.\;\\
 \end{array}\right.$$
Let $Y=(y(v_{0})$, $y(v_{1})$, $\ldots$, $y(v_{n-1}))^{T}=X+F$. Then $y(v_{0})=0$, $y(v_{\frac{k+j}{2}})=0$, $y(v_{i})=x(v_{i})-x(v_{k+j-i})=-(x(v_{k+j-i})-x(v_{i}))=-y(v_{k+j-i})$ for $j\leq i\leq \frac{k+j}{2}-1$, and $y(v_{\eta_{i}})=x(v_{\eta_{i}})- x(v_{\eta_{z-i+1}})=-(x(v_{\eta_{z-i+1}})-x(v_{\eta_{i}}))=- y(v_{\eta_{z-i+1}})$ for $1\leq i\leq\frac{z}{2}$ if $z\geq2$ now. Note the condition that there is a $x(v_{i})$ for $j\leq i<\frac{k+j}{2}$ that $x(v_{i})\neq x(v_{k+j-i})$, or there is a $x(v_{\eta_{i}})$ for $1\leq i<\frac{z+1}{2}$ that $x(v_{\eta_{i}})\neq x(v_{\eta_{z-i+1}})$. Thus $Y\neq \mathbf{0}^{T}$. Note the symmetry of $\Theta(j,k)$. Thus $F$ is also an eigenvector of $\Theta(j,k)$ corresponding to $q(\Theta(j,k))$. Then $Y$ is also an eigenvector of $\Theta(j,k)$ corresponding to $q(\Theta(j,k))$.

Now, we prove $y(v_{\frac{k+j}{2}-1})\neq 0$.
Otherwise, suppose that $y(v_{\frac{k+j}{2}-1})= 0$. Thus if $\frac{k+j}{2}\geq j+2$, from $q(\Theta(j,k))y(v_{\frac{k+j}{2}-1})=2y(v_{\frac{k+j}{2}-1})+y(v_{\frac{k+j}{2}-2})+y(v_{\frac{k+j}{2}})$, then it follows that $y(v_{\frac{k+j}{2}-2})= 0$. If $\frac{k+j}{2}\geq j+3$, by induction, from $q(\Theta(j,k))y(v_{i})=2y(v_{i})+y(v_{i-1})+y(v_{i+1})$ for $j+1\leq i\leq \frac{k+j}{2}-2$, then it follows that $y(v_{s})=0$ for $j\leq s\leq \frac{k+j}{2}-3$. From $q(\Theta(j,k))y(v_{0})=2y(v_{0})+y(v_{j})+y(v_{k})$ and $y(v_{j})=0$, then we have $y(v_{0})=0$. If $k-j\leq n-4$, from $q(\Theta(j,k))y(v_{j})=3y(v_{j})+y(v_{0})+y(v_{j+1})+y(v_{\eta_{1}})$, then $y(v_{\eta_{1}})=0$ follows. Proceeding like this, we get that $y(v_{\eta_{i}})=0$ for $1\leq i\leq z$. Then $y=\mathbf{0}^{T}$, which contradicts that $Y\neq \mathbf{0}^{T}$. Thus $y(v_{\frac{k+j}{2}-1})\neq 0$.

Consequently, $Y$ makes the lemma hold. Similarly, if $\|V(P_{2}(\Theta(j,k)))\|$ is odd, we can get an analogous eigenvector $Y$ making the lemma hold.
 \ \ \ \ \ $\Box$
\end{proof}

\begin{lemma}\label{le03.07} 
Suppose both $k$ and $j$ are positive integers, and $k+j\geq 2$ is even. If any eigenvector of $\Theta(j,k)$ corresponding to $q(\Theta(j,k))$ does not satisfy the conclusion in Lemma \ref{le03.06} (that is, there is no eigenvector of $\Theta(j,k)$ corresponding to $q(\Theta(j,k))$ satisfying the conclusion in Lemma \ref{le03.06}), then any eigenvector $W=(w(v_{0})$, $w(v_{1})$, $\ldots$, $w(v_{n-1}))^{T}\in R^{n}$ of $\Theta(j,k)$ corresponding to $q(\Theta(j,k))$ satisfies that

(1) $w(v_{i})=w(v_{k+j-i})$ for $j\leq i<\frac{k+j}{2}$;

(2) $w(v_{\eta_{i}})=w(v_{\eta_{z-i+1}})$ for $1\leq i<\frac{z+1}{2}$ if $z\geq 2$;

(3) $w(v_{i})\neq 0$ for $0\leq i\leq n-1$;

(4) $w(v_{i})w(v_{l})< 0$ for any edge $v_{i}v_{l}\neq v_{\eta_{\frac{z}{2}}}v_{\eta_{\frac{z}{2}+1}}$;

(5) $\mid w(v_{i})|> \mid w(v_{i-1})|$ for $j+1\leq i\leq \frac{k+j}{2}$;

(6) $\mid w(v_{0})|>\mid w(v_{j})|$;

(7) $\mid w(v_{j})|>\mid w(v_{\eta_{1}})|$ if $z\geq2$;

(8) $\mid w(v_{\eta_{i-1}})|>\mid w(v_{\eta_{i}})|$ for $2\leq i\leq \frac{z}{2}$ if $z\geq4$.

\end{lemma}

\begin{proof}
Suppose $W=(w(v_{0})$, $w(v_{1})$, $\ldots$, $w(v_{n-1}))^{T}\in R^{n}$ is an eigenvector corresponding to $q(\Theta(j,k))$. (1), (2) follow from Lemma \ref{le03.06} as corollaries directly. Note that $deg(v_{0})=2$, $\Theta(j,k)$ is connected and nonbipartite. Combining Lemma \ref{le02,02}, we have $0<q(\Theta(j,k))< 2$. Next, we prove (3)- (6). Note that both $n$ and $k+j$ are even. Then $k-j+1$ is odd, and $\|V(P_{2}(\Theta(j,k)))\|$ is even.

Let $H=\Theta(j,k)-v_{\eta_{\frac{z}{2}}}v_{\eta_{\frac{z}{2}+1}}$ if $z\geq 2$, and $H=\Theta(j,k)-v_{j}v_{k}$ if $P_{2}(\Theta(j,k))=v_{j}v_{k}$. Let $F=(f(v_{0})$, $f(v_{1})$, $\ldots$, $f(v_{n-1}))^{T}\in R^{n}$ be a new vector satisfying that $f(v_{i})=(-1)^{dist_{H}(v_{\eta_{\frac{z}{2}}}, v_{i})}\mathrm{sgn}(w(v_{\eta_{\frac{z}{2}}}))| w(v_{i})|$ for any $v_{i}\in V(\Theta(j,k))$ if $z\geq 2$; $f(v_{i})=(-1)^{dist_{H}(v_{j}, v_{i})}\mathrm{sgn}(w(v_{j}))| w(v_{i})|$ for any $v_{i}\in V(\Theta(j,k))$ if $P_{2}(\Theta(j,k))=v_{j}v_{k}$. Then $F$ satisfies

(i) $f(v_{i})=f(v_{k+j-i})$ for $j\leq i\leq\frac{k+j}{2}$;

(ii) $f(v_{\eta_{i}})=f(v_{\eta_{z-i+1}})$ for $1\leq i\leq\frac{z}{2}$ if $z\geq 2$;

(iii) $f(v_{i})f(v_{l})\leq0$ for any edge $v_{i}v_{l}\neq v_{\eta_{\frac{z}{2}}}v_{\eta_{\frac{z}{2}+1}}$.

$F$ also satisfies that $\frac{F^{T}Q(\Theta(j,k))F}{F^{T}F}\leq \frac{W^{T}Q(\Theta(j,k))W}{W^{T}W}=q(\Theta(j,k))$. Note that $W$ is an eigenvector corresponding to $q(\Theta(j,k))$. Then both $W$ and $F$ are not zero vector. Thus $F$ is also an eigenvector of $\Theta(j,k)$ corresponding to $q(\Theta(j,k))$.

{\bf Assertion 1} $f(v_{\frac{k+j}{2}})\neq 0$. Otherwise, suppose that $f(v_{\frac{k+j}{2}})= 0$. Then from $q(\Theta(j,k))f(v_{\frac{k+j}{2}})=2f(v_{\frac{k+j}{2}})+f(v_{\frac{k+j}{2}-1})+f(v_{\frac{k+j}{2}+1})$, it follows that $f(v_{\frac{k+j}{2}-1})= f(v_{\frac{k+j}{2}+1})=0$. By induction, from $q(\Theta(j,k))f(v_{i})=2f(v_{i})+f(v_{i-1})+f(v_{i+1})$ for $j+1\leq i\leq \frac{k+j}{2}-1$, we get that $f(v_{s})=0$ for $j\leq s\leq \frac{k+j}{2}-2$. From $q(\Theta(j,k))f(v_{0})=2f(v_{0})+f(v_{j})+f(v_{k})$, we get $f(v_{0})=0$. If $k-j\leq n-4$, from $q(\Theta(j,k))f(v_{j})=3f(v_{j})+f(v_{0})+f(v_{j+1})+f(v_{\eta_{1}})$, then it follows that $f(v_{\eta_{1}})=0$. Proceeding like this, we get that $f(v_{\eta_{i}})=0$ for $1\leq i\leq z$. Then $F=\mathbf{0}^{T}$, which contradicts that $F\neq \mathbf{0}^{T}$. Thus $f(v_{\frac{k+j}{2}})\neq 0$. Then Assertion 1 follows.

Note that if $f(v_{\frac{k+j}{2}-1})=0$, from $q(\Theta(j,k))f(v_{\frac{k+j}{2}})=2f(v_{\frac{k+j}{2}})+f(v_{\frac{k+j}{2}-1})+f(v_{\frac{k+j}{2}+1})$, we get $f(v_{\frac{k+j}{2}})=0$ which contradicts Assertion 1. Then we get the following Assertion 2.

{\bf Assertion 2} $f(v_{\frac{k+j}{2}-1})\neq 0$.

{\bf Assertion 3} $| f(v_{\frac{k+j}{2}})|> | f(v_{\frac{k+j}{2}-1})|$. We prove it by contradiction. Suppose that $\mid f(v_{\frac{k+j}{2}})|\leq \mid f(v_{\frac{k+j}{2}-1})|$. Note that from the above definition of $F$, Assertion 1 and Assertion 2, we have $f(v_{\frac{k+j}{2}-1})= f(v_{\frac{k+j}{2}+1})\neq 0$, $f(v_{\frac{k+j}{2}})\neq 0$ and $f(v_{\frac{k+j}{2}})f(v_{\frac{k+j}{2}-1})< 0$. From $q(\Theta(j,k))f(v_{\frac{k+j}{2}})=2f(v_{\frac{k+j}{2}})+f(v_{\frac{k+j}{2}-1})+f(v_{\frac{k+j}{2}+1})$, we get that $q(\Theta(j,k))\leq0$, which contradicts $q(\Theta(j,k))>0$. Therefore, $\mid f(v_{\frac{k+j}{2}})|> | f(v_{\frac{k+j}{2}-1})|$ follows.

{\bf Assertion 4} $| f(v_{i})|> | f(v_{i-1})|$ and $| f(v_{i})\mid > 0$ for $j+1\leq i\leq \frac{k+j}{2}$.
Note that $f(v_{\frac{k+j}{2}-1})\neq 0$, $| f(v_{\frac{k+j}{2}})|> | f(v_{\frac{k+j}{2}-1})|$, and $f(v_{i})f(v_{i-1})\leq 0$ for $j+1\leq i\leq \frac{k+j}{2}$. If $| f(v_{\frac{k+j}{2}-1})|\leq | f(v_{\frac{k+j}{2}-2})|$, from $q(\Theta(j,k))f(v_{\frac{k+j}{2}-1})=2f(v_{\frac{k+j}{2}-1})+f(v_{\frac{k+j}{2}})+f(v_{\frac{k+j}{2}-2})$, then it follows that $q(\Theta(j,k))<0$,
which contradicts $q(\Theta(j,k))>0$. Thus we get $| f(v_{\frac{k+j}{2}-1})|> | f(v_{\frac{k+j}{2}-2})|$. Similarly, from $q(\Theta(j,k))f(v_{i})=2f(v_{i})+f(v_{i+1})+f(v_{i-1})$ for $j+1\leq i\leq \frac{k+j}{2}-2$, by induction, we get that $| f(v_{i})|> | f(v_{i-1})|$ and $| f(v_{i})\mid > 0$ for $j+1\leq i\leq \frac{k+j}{2}-2$. Combined Assertion 3, Assertion 4 follows.

{\bf Assertion 5} If $z\geq 2$, then $f(v_{\eta_{\frac{z}{2}}})\neq 0$. We prove it by contradiction. Note that $z$ is even. Suppose $f(v_{\eta_{\frac{z}{2}}})=0$. Note that $f(v_{\eta_{\frac{z}{2}}})=f(v_{\eta_{\frac{z}{2}+1}})$, $q(\Theta(j,k))f(v_{\eta_{\frac{z}{2}}})=2f(v_{\eta_{\frac{z}{2}}})+f(v_{\eta_{\frac{z}{2}-1}})+f(v_{\eta_{\frac{z}{2}+1}})$. Thus $f(v_{\eta_{\frac{z}{2}-1}})=0$. By induction, from $q(\Theta(j,k))f(v_{\eta_{i}})=2f(v_{\eta_{i}})+f(v_{\eta_{i-1}})+f(v_{\eta_{i+1}})$ for $2\leq i\leq \frac{z}{2}-1$ and $q(\Theta(j,k))f(v_{\eta_{1}})=2f(v_{\eta_{1}})+f(v_{j})+f(v_{\eta_{2}})$, it follows that $f(v_{\eta_{i}})=0$ for $1\leq i\leq \frac{z}{2}-2$ and $f(v_{j})=0$.
From $q(\Theta(j,k))f(v_{0})=2f(v_{0})+f(v_{j})+f(v_{k})$, it follows that $f(v_{0})=0$. From $q(\Theta(j,k))f(v_{j})=3f(v_{j})+f(v_{j+1})+f(v_{0})+f(v_{\eta_{1}})$, we get $f(v_{j+1})=0$, which contradicts that $|f(v_{j+1})| >0$ in Assertion 4. As a result, $f(v_{\eta_{\frac{z}{2}}})\neq 0$ follows.

Similarly, we get the following Assertion 6.

{\bf Assertion 6} If $k-j= n-2$, then $f(v_{j})\neq 0$.

{\bf Assertion 7} If $z\geq 2$, then $f(v_{\eta_{i}})\neq 0$ for $1\leq i\leq \frac{z}{2}-1$. We prove this assertion by contradiction. Suppose that $f(v_{\eta_{s}})= 0$ for some $2\leq s\leq \frac{z}{2}-1$. From $q(\Theta(j,k))f(v_{\eta_{s}})=2f(v_{\eta_{s}})+f(v_{\eta_{s-1}})+f(v_{\eta_{s+1}})$, noting that $f(v_{\eta_{s-1}})f(v_{\eta_{s+1}})\geq 0$, we get $f(v_{\eta_{s-1}})=f(v_{\eta_{s+1}})= 0$. Similar to Assertion 5, we get that $f(v_{\eta_{i}})=0$ for $1\leq i\leq \frac{z}{2}$, $f(v_{j})=0$ and $f(v_{j+1})=0$, which contradicts that $|f(v_{j+1})|>0$ in Assertion 4. Thus the assertion holds.

{\bf Assertion 8} $f(v_{j})\neq0$. Note that $f(v_{j+1})\neq0$, $f(v_{\eta_{1}})\neq 0$, $f(v_{j+1})f(v_{j})\leq 0$, $f(v_{\eta_{1}})f(v_{j})\leq 0$ if $z\geq 2$, $f(v_{0})f(v_{j})\leq 0$, $f(v_{j+1})f(v_{0})\geq 0$, $f(v_{j+1})f(v_{\eta_{1}})\geq 0$ and $f(v_{0})f(v_{\eta_{1}})\geq 0$ if $z\geq 2$. From $q(\Theta(j,k))f(v_{j})=3f(v_{j})+f(v_{j+1})+f(v_{0})+f(v_{\eta_{1}})$ if $z\geq 2$, we get $f(v_{j})\neq 0$; from $q(\Theta(j,k))f(v_{j})=3f(v_{j})+f(v_{j+1})+f(v_{0})+f(v_{k})$ if $k-j= n-2$ and $f(v_{j})=f(v_{k})$, we get $f(v_{j})\neq 0$. Then this assertion follows.

{\bf Assertion 9} $| f(v_{0})|>| f(v_{j})|$. This assertion follows from the facts that $f(v_{j})=f(v_{k})$, $f(v_{0})f(v_{j})\leq 0$, $q(\Theta(j,k))>0$ and $q(\Theta(j,k))f(v_{0})=2f(v_{0})+f(v_{j})+f(v_{k})$.

Note that $| f(v_{0})|>| f(v_{j})|$, $| f(v_{j+1})|>| f(v_{j})|$. From that $q(\Theta(j,k))f(v_{j})=3f(v_{j})+f(v_{j+1})+f(v_{0})+f(v_{\eta_{1}})$ if $z\geq 2$, $q(\Theta(j,k))f(v_{\eta_{1}})=2f(v_{\eta_{1}})+f(v_{j})+f(v_{\eta_{2}})$ if $z\geq 4$, $q(\Theta(j,k))f(v_{\eta_{i}})=2f(v_{\eta_{i}})+f(v_{\eta_{i-1}})+f(v_{\eta_{i+1}})$ for $2\leq i\leq \frac{z}{2}-1$ if $z\geq 6$, as Assertion 4, we get the following Assertion 10.

{\bf Assertion 10} $| f(v_{j})|>| f(v_{\eta_{1}})|$ if $z\geq 2$, $| f(v_{\eta_{i-1}})|>| f(v_{\eta_{i}})|$ for $2\leq i\leq \frac{z}{2}$ if $z\geq 4$.

From the above Assertions 1-11, we get the following Assertion 11.

{\bf Assertion 11} $f(v_{i})\neq 0$ for $0\leq i\leq n-1$.

{\bf Assertion 12} $W=F$. Note that $|f(v_{i})|=|w(v_{i})|$ for $0\leq i\leq n-1$, $f(v_{j})=f(v_{k})=w(v_{j})=w(v_{k})$ if $P_{2}(\Theta(j,k))=v_{j}v_{k}$, and $f(v_{\eta_{\frac{z}{2}}})=f(v_{\eta_{\frac{z}{2}+1}})=w(v_{\eta_{\frac{z}{2}}})=w(v_{\eta_{\frac{z}{2}+1}})$ if $z\geq 2$. Assume that $W\neq F$. Then there is $\mathrm{sgn}(w(v_{i}))\neq (-1)^{dis_{H}(v_{\eta_{\frac{z}{2}}}, v_{i})}\mathrm{sgn}(w(v_{\eta_{\frac{z}{2}}}))$ for some $v_{i}\not\in \{v_{\eta_{\frac{z}{2}}}, v_{\eta_{\frac{z}{2}+1}}\}$ if $z\geq 2$; there is $\mathrm{sgn}(w(v_{i}))\neq (-1)^{dis_{H}(v_{j}, v_{i})}\mathrm{sgn}(w(v_{j}))$ for some $v_{i}\not\in \{v_{j}, v_{k}\}$ if $P_{2}(\Theta(j,k))=v_{j}v_{k}$.
Suppose $v_{s}$ is the nearest vertex from $v_{\eta_{\frac{z}{2}}}$ that $\mathrm{sgn}(w(v_{i}))\neq (-1)^{dis_{H}(v_{\eta_{\frac{z}{2}}}, v_{s})}\mathrm{sgn}(w(v_{\eta_{\frac{z}{2}}}))$ if $z\geq 2$; $v_{s}$ is the nearest vertex from $v_{j}$ that $\mathrm{sgn}(w(v_{i}))\neq (-1)^{dis_{H}(v_{j}, v_{s})}\mathrm{sgn}(w(v_{j}))$ if $P_{2}(\Theta(j,k))=v_{j}v_{k}$. And suppose $\mathcal {P}_{s}$ is a shortest path from $v_{\eta_{\frac{z}{2}}}$ to $v_{s}$ if $z\geq 2$; $\mathcal {P}_{s}$ is a shortest path from $v_{j}$ to $v_{s}$ if $P_{2}(\Theta(j,k))=v_{j}v_{k}$. Furthermore, we suppose $v_{\alpha}$ is the vertex adjacent to $v_{s}$ along $\mathcal {P}_{s}$. Then $\mathrm{sgn}(w(v_{s}))=\mathrm{sgn}(w(v_{\alpha}))$ and $(f(v_{s})+f(v_{\alpha}))^{2}=(w(v_{s})-w(v_{\alpha}))^{2}<(w(v_{s})+w(v_{\alpha}))^{2}$. Thus $F^{T}Q(\Theta(j,k))F< W^{T}Q(\Theta(j,k))W$. Note that $F^{T}F=W^{T}W$. Then we get a contradiction that $q(\Theta(j,k))=\frac{F^{T}Q(\Theta(j,k))F}{F^{T}F}<\frac{W^{T}Q(\Theta(j,k))W}{W^{T}W}=q(\Theta(j,k))$. Thus the assertion holds.

Consequently, from Assertions 1-12, we get that $W$ satisfies (3)-(8). Note the arbitrariness of $W$. Then the lemma follows.
 \ \ \ \ \ $\Box$
\end{proof}

\begin{lemma}\label{le03.08} 
If there is an eigenvector $Y=(y(v_{0})$, $y(v_{1})$, $\ldots$, $y(v_{n-1}))^{T}\in R^{n}$ corresponding to $q(\Theta(j,k))$ satisfying the conclusion in Lemma \ref{le03.06},  then $q(\Theta)\leq q(\Theta(j,k))$ with equality if and only if $\Theta(j,k)\cong \Theta$.
\end{lemma}

\begin{proof}
Suppose $|y(v_{\mu})|=\max\{|y(v_{i})|\mid v_{i}\in V(\Theta(j,k))\}$. Note that $y(v_{0})=0$. Then $v_{\mu}$ is in the odd cycle $\mathscr{C}$. Note that one of $\|V(P_{1}(\Theta(j,k)))\|$ and $\|V(P_{2}(\Theta(j,k)))\|$ is odd. Without loss of generality, for convenience, we suppose $\|V(P_{1}(\Theta(j,k)))\|$ is odd (for the case that $\|V(P_{2}(\Theta(j,k)))\|$ is odd, it is proved similarly). Then both $k+j$ and  $\|V(P_{2}(\Theta(j,k)))\|$ are even. If $\frac{k+j}{2}=j+1$, then the lemma is trivial because $\Theta(j,k)\cong \Theta$ now. Thus we suppose $\frac{k+j}{2}\geq j+2$ next. And then, we consider two cases.

{\bf Case 1} $| y(v_{\frac{k+j}{2}-1})|\leq \mid y(v_{j})|$. Let $G^{'}=\Theta(j,k)-v_{0}v_{j}-v_{0}v_{k}+v_{0}v_{\frac{k+j}{2}-1}+v_{0}v_{\frac{k+j}{2}+1}$. Note that $y(v_{0})=0$. Then $Y^{T}Q(G^{'})Y\leq Y^{T}Q(\Theta(j,k))Y$ and $q(G^{'})\leq\frac{Y^{T}Q(G^{'})Y}{Y^{T}Y}\leq \frac{Y^{T}Q(\Theta(j,k))Y}{Y^{T}Y}=q(\Theta(j,k))$.

We assert that $q(G^{'})<q(\Theta(j,k))$. Otherwise, suppose $q(G^{'})=q(\Theta(j,k))$. Then $Y$ is also a eigenvector of $G^{'}$.
By Lemma \ref{le03.06}, we know that $y(v_{\frac{k+j}{2}-1})\neq0$. Note that $q(G^{'})y(v_{\frac{k+j}{2}-1})=3y(v_{\frac{k+j}{2}-1})+y(v_{\frac{k+j}{2}})+y(v_{\frac{k+j}{2}-2})+y(v_{0})$, $q(\Theta(j,k))y(v_{\frac{k+j}{2}-1})=2y(v_{\frac{k+j}{2}-1})+y(v_{\frac{k+j}{2}})+y(v_{\frac{k+j}{2}-2})$. Consequently, it follows that $q(G^{'})y(v_{\frac{k+j}{2}-1})\neq q(\Theta(j,k))y(v_{\frac{k+j}{2}-1})$, and then $q(G^{'})\neq q(\Theta(j,k))$, which contradicts the supposition $q(G^{'})=q(\Theta(j,k))$. As a result, we get that $q(G^{'})<q(\Theta(j,k))$. Then our assertion holds.

{\bf Case 2} $| y(v_{\frac{k+j}{2}-1})|> \mid y(v_{j})|$. Then $v_{\mu}\neq v_{\frac{k+j}{2}}$, $v_{\mu}\neq v_{j}$ and $v_{\mu}\neq v_{k}$ follow immediately. Let $G_{1}=\Theta(j,k)+v_{\frac{k+j}{2}-1}v_{\frac{k+j}{2}+1}$ (see Fig. 3.1). Then $q(G_{1})\geq q(\Theta(j,k))$ by Lemma \ref{le02,03}. Note that $q(G_{1})\leq \frac{Y^{T}Q(G_{1})Y}{Y^{T}Y}=\frac{Y^{T}Q(\Theta(j,k))Y}{Y^{T}Y}=q(\Theta(j,k))$. Then we get that $q(G_{1})= \frac{Y^{T}Q(G_{1})Y}{Y^{T}Y}=q(\Theta(j,k))$, and that $Y$ is an eigenvector of $G_{1}$ corresponding to $q(G_{1})$.

{\bf Subcase 2.1} $v_{\mu}$ is in $P_{1}(\Theta(j,k))$. Now, note that $y(v_{i})=-y(v_{k+j-i})$ for $j\leq i\leq \frac{k+j}{2}-1$, $y(v_{\frac{k+j}{2}})=0$, $y(v_{\frac{k+j}{2}-1})\neq0$. Thus without loss of generality, we can assume $j< \mu\leq \frac{k+j}{2}-1$.

\setlength{\unitlength}{0.6pt}
\begin{center}
\begin{picture}(415,452)
\put(40,395){\circle*{4}}
\put(100,363){\circle*{4}}
\qbezier(40,395)(70,379)(100,363)
\put(100,428){\circle*{4}}
\qbezier(40,395)(70,412)(100,428)
\put(409,428){\circle*{4}}
\put(409,361){\circle*{4}}
\qbezier(409,428)(409,395)(409,361)
\put(3,397){$v_{\frac{k+j}{2}}$}
\put(80,351){$v_{\frac{k+j}{2}-1}$}
\put(412,430){$v_{\eta_{2}}$}
\put(80,440){$v_{\frac{k+j}{2}+1}$}
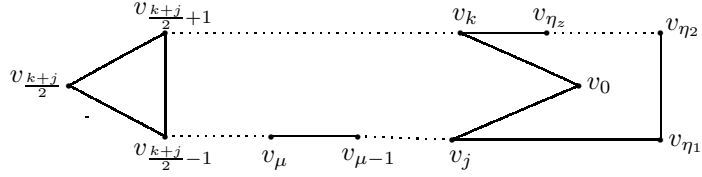
\put(190,309){$G_{1}$ in Case 2}
\put(273,239){\circle*{4}}
\put(62,269){\circle*{4}}
\put(62,203){\circle*{4}}
\qbezier(62,269)(62,236)(62,203)
\put(335,269){\circle*{4}}
\put(335,203){\circle*{4}}
\put(395,241){\circle*{4}}
\qbezier(335,269)(365,255)(395,241)
\qbezier(335,203)(365,222)(395,241)
\put(236,242){$v_{\frac{k+j}{2}}$}
\put(40,191){$v_{\frac{k+j}{2}-1}$}
\put(400,239){$v_{0}$}
\put(42,282){$v_{\frac{k+j}{2}+1}$}
\put(180,155){$G^{'}$ in Subcase 2.1.1}
\put(180,-2){$G^{'}$ in Subcase 2.2}

\put(178,-35){Fig. 3.1. $G_{1}$ and $G^{'}$}
\put(330,190){$v_{j}$}
\put(220,363){\circle*{4}}
\put(284,428){\circle*{4}}
\put(166,363){\circle*{4}}
\put(338,428){\circle*{4}}
\put(279,361){\circle*{4}}
\qbezier(51,375)(52,375)(51,375)
\put(160,347){$v_{\mu}$}
\put(210,348){$v_{\mu-1}$}
\put(277,347){$v_{j}$}
\put(279,436){$v_{k}$}
\put(414,357){$v_{\eta_{1}}$}
\put(329,436){$v_{\eta_{z}}$}
\put(358,395){\circle*{4}}
\qbezier(284,428)(321,412)(358,395)
\put(363,393){$v_{0}$}
\qbezier(358,395)(318,378)(279,361)
\qbezier(100,428)(100,396)(100,363)
\put(268,203){\circle*{4}}
\put(113,203){\circle*{4}}
\qbezier(51,237)(52,237)(51,237)
\qbezier(51,237)(51,237)(52,237)
\qbezier(335,269)(304,254)(273,239)
\qbezier(335,203)(304,221)(273,239)
\put(332,276){$v_{k}$}
\put(107,190){$v_{\mu}$}
\put(260,190){$v_{\mu-1}$}
\put(163,203){\circle*{4}}
\qbezier(113,203)(138,203)(163,203)
\put(222,203){\circle*{4}}
\qbezier(268,203)(245,203)(222,203)
\put(152,190){$v_{\eta_{1}}$}
\put(212,190){$v_{\eta_{z}}$}
\put(107,44){\circle*{4}}
\put(107,111){\circle*{4}}
\put(399,44){\circle*{4}}
\put(287,44){\circle*{4}}
\put(341,44){\circle*{4}}
\put(47,78){\circle*{4}}
\put(47,78){\circle*{4}}
\qbezier(47,78)(48,78)(47,78)
\put(47,78){\circle*{4}}
\put(47,78){\circle*{4}}
\qbezier(47,78)(48,78)(47,78)
\put(47,78){\circle*{4}}
\qbezier(107,111)(77,95)(47,78)
\put(47,78){\circle*{4}}
\qbezier(107,44)(77,61)(47,78)
\put(47,78){\circle*{4}}
\put(47,78){\circle*{4}}
\qbezier(47,78)(48,78)(47,78)
\put(47,78){\circle*{4}}
\qbezier(107,111)(77,95)(47,78)
\put(166,111){\circle*{4}}
\qbezier(107,111)(136,111)(166,111)
\put(165,44){\circle*{4}}
\qbezier(107,44)(136,44)(165,44)
\put(220,44){\circle*{4}}
\qbezier(165,44)(192,44)(220,44)
\put(399,111){\circle*{4}}
\qbezier(399,44)(399,78)(399,111)
\put(404,112){$v_{\eta_{2}}$}
\put(160,119){$v_{\eta_{z}}$}
\put(101,31){$v_{j}$}
\put(103,118){$v_{k}$}
\put(154,33){$v_{\eta_{1}}$}
\put(401,32){$v_{j+1}$}
\put(212,31){$v_{k-1}$}
\qbezier(287,44)(314,44)(341,44)
\put(260,31){$v_{\frac{k+j}{2}+1}$}
\put(324,31){$v_{\frac{k+j}{2}-1}$}
\qbezier[46](166,111)(282,111)(399,111)
\qbezier(284,428)(311,428)(338,428)
\qbezier[36](100,428)(192,428)(284,428)
\qbezier[14](338,428)(373,428)(409,428)
\qbezier(166,363)(193,363)(220,363)
\qbezier[13](100,363)(133,363)(166,363)
\qbezier[11](220,363)(249,362)(279,361)
\qbezier(62,269)(63,269)(62,269)
\qbezier[54](62,269)(198,269)(335,269)
\qbezier[10](62,203)(87,203)(113,203)
\qbezier[11](163,203)(192,203)(222,203)
\qbezier[13](268,203)(301,203)(335,203)
\qbezier[13](220,44)(253,44)(287,44)
\qbezier[11](341,44)(370,44)(399,44)
\put(174,81){\circle*{4}}
\qbezier(107,111)(140,96)(174,81)
\qbezier(174,81)(140,63)(107,44)
\put(181,80){$v_{0}$}
\put(10,80){$v_{\frac{k+j}{2}}$}
\qbezier(279,361)(344,361)(409,361)
\end{picture}
\end{center}

\

{\bf Subcase 2.1.1} $z\geq 2$. Then we let
$G^{'}=G_{1}-v_{\frac{k+j}{2}-1}v_{\frac{k+j}{2}}-v_{\frac{k+j}{2}+1}v_{\frac{k+j}{2}}+v_{j}v_{\frac{k+j}{2}}+v_{k}v_{\frac{k+j}{2}}-v_{j}v_{\eta_{1}}-v_{\eta_{z}}v_{k}-v_{\mu}v_{\mu-1}+v_{\mu}v_{\eta_{1}}+v_{\eta_{z}}v_{\mu-1}$ (see Fig. 3.1).
Now, let $F=(f(v_{0})$, $f(v_{1})$, $\ldots$, $f(v_{n-1}))^{T}$ be a vector satisfying that
$$\left \{\begin{array}{ll}
 f(v_{i})=y(v_{i}),\ & \ v_{i}\in V(\Theta(j,k))\setminus \{v_{\eta_{1}}, v_{\eta_{2}}, \cdots, v_{\eta_{z}}\};\\
\\ f(v_{\eta_{i}})=(-1)^{i}y(v_{\mu}),\ & \ i=1, 2, \ldots, z. \end{array}\right.$$
Then $F^{T}F\geq Y^{T}Y$ and $F^{T}Q(G^{'})F\leq Y^{T}Q(G_{1})Y= Y^{T}Q(\Theta(j,k))Y$. As a result, we have $q(G^{'})\leq\frac{F^{T}Q(G^{'})F}{F^{T}F}\leq \frac{Y^{T}Q(\Theta(j,k))Y}{Y^{T}Y}=q(\Theta(j,k))$. As proved for $q(\mathbb{C})<q(G)$ in the proof of Lemma \ref{le03.04} (by considering $q(G^{'})f(v_{\eta_{1}})$), we get that $q(G^{'})<q(\Theta(j,k))$.

{\bf Subcase 2.1.2}  $P_{2}(\Theta(j,k))= v_{j}v_{k}$. Then we let
$G^{'}=G_{1}-v_{\frac{k+j}{2}-1}v_{\frac{k+j}{2}}-v_{\frac{k+j}{2}+1}v_{\frac{k+j}{2}}+v_{j}v_{\frac{k+j}{2}}+v_{k}v_{\frac{k+j}{2}}-v_{j}v_{k}$.
Then $Y^{T}Q(G^{'})Y\leq Y^{T}Q(G_{1})Y= Y^{T}Q(\Theta(j,k))Y$ and $q(G^{'})\leq\frac{Y^{T}Q(G^{'})Y}{Y^{T}Y}\leq \frac{Y^{T}Q(G_{1})Y}{Y^{T}Y}= \frac{Y^{T}Q(\Theta(j,k))Y}{Y^{T}Y}=q(\Theta(j,k))$. Note that $q(G_{1})=q(\Theta(j,k))$. As Case 1, by considering $q(G^{'})y(v_{\frac{k+j}{2}-1})$ and $q(G_{1})y(v_{\frac{k+j}{2}-1})$, we get that $q(G^{'})<q(G_{1})$, and $q(G^{'})<q(\Theta(j,k))$ further.

{\bf Subcase 2.2}  $\mu\in \{\eta_{1}$, $\eta_{2}$, $\ldots$, $\eta_{z}\}$ and $z\geq 2$. Without loss of generality, for convenience, suppose $\mu=\eta_{2}$. Then we let $G^{'}=G_{1}-v_{\frac{k+j}{2}-1}v_{\frac{k+j}{2}}-v_{\frac{k+j}{2}+1}v_{\frac{k+j}{2}}+v_{j}v_{\frac{k+j}{2}}+v_{k}v_{\frac{k+j}{2}}-v_{j}v_{j+1}-v_{k-1}v_{k}-v_{\eta_{2}}v_{\eta_{1}}+v_{k-1}v_{\eta_{1}}+v_{j+1}v_{\eta_{2}}$
(where if $\mu=\eta_{1}$, we can let $G^{'}=G_{1}-v_{\frac{k+j}{2}-1}v_{\frac{k+j}{2}}-v_{\frac{k+j}{2}+1}v_{\frac{k+j}{2}}+v_{j}v_{\frac{k+j}{2}}+v_{k}v_{\frac{k+j}{2}}-v_{j}v_{j+1}-v_{k-1}v_{k}-v_{\eta_{2}}v_{\eta_{1}}+v_{\eta_{1}}v_{j+1}+v_{k-1}v_{\eta_{2}}$) (see Fig. 3.1).

Now, let $F=(f(v_{0})$, $f(v_{1})$, $\ldots$, $f(v_{n-1}))^{T}$ be a vector satisfying that
$$\left \{\begin{array}{ll}
 f(v_{i})=y(v_{i}),\ & \ v_{i}\in V(\Theta(j,k))\setminus \{v_{j+1}, v_{j+2}, \cdots, v_{\frac{k+j}{2}-1}, v_{\frac{k+j}{2}+1}, v_{\frac{k+j}{2}+2}, \ldots, v_{k-1}\};\\
\\ f(v_{i})=(-1)^{i-j}y(v_{\mu}),\ & \ i=j+1, j+2, \ldots, \frac{k+j}{2}-1;\\
\\ f(v_{i})=(-1)^{i-j-1}y(v_{\mu}),\ & \ i=\frac{k+j}{2}+1, \frac{k+j}{2}+2, \ldots, k-1. \end{array}\right.$$
Note that $\frac{k+j}{2}-1>j$. As Subcase 2.1.1 (by considering $q(G^{'})f(v_{j+1})$), we get that $q(G^{'})<q(\Theta(j,k))$.

Note that $G^{'}\cong \Theta$ for the above $G^{'}$ in Case 1 and Case 2. Consequently, this lemma follows from the above discussion.
 \ \ \ \ \ $\Box$
\end{proof}

\begin{lemma}\label{le03.09} 
If any eigenvector of $\Theta(j,k)$ corresponding to $q(\Theta(j,k))$ does not satisfy the conclusion in Lemma \ref{le03.06}, then $q(\Theta)\leq q(\Theta(j,k))$ with equality if and only if $\Theta(j,k)\cong \Theta$.

\end{lemma}

\begin{proof}
Without loss of generality, for convenience, we suppose $\|V(P_{1}(\Theta(j,k)))\|$ is odd (for the case that $\|V(P_{2}(\Theta(j,k)))\|$ is odd, it is proved similarly). Then both $k+j$ and  $\|V(P_{2}(\Theta(j,k)))\|$ are even. If $\frac{k+j}{2}=j+1$, then the lemma is trivial because $\Theta(j,k)\cong \Theta$ now. Thus we suppose $\frac{k+j}{2}\geq j+2$ next.

By Lemma \ref{le03.07}, an eigenvector $W=(w(v_{0})$, $w(v_{1})$, $\ldots$, $w(v_{n-1}))^{T}\in R^{n}$ of $\Theta(j,k)$ corresponding to $q(\Theta(j,k))$ must satisfy the conclusions (1)-(8) of  Lemma \ref{le03.07}.

Let $G^{'}=\Theta(j,k)-v_{0}v_{j}-v_{0}v_{k}+v_{0}v_{\frac{k+j}{2}-1}+v_{0}v_{\frac{k+j}{2}+1}$. Note that $w(v_{i})\neq 0$ for $0\leq i\leq n-1$, $w(v_{\frac{k+j}{2}-1})=w(v_{\frac{k+j}{2}+1})$, $w(v_{j})=w(v_{k})$, $| w(v_{\frac{k+j}{2}-1})|> | w(v_{j})|$, $| w(v_{0})|>| w(v_{j})|$, $w(v_{0})w(v_{j})< 0$, and $w(v_{\frac{k+j}{2}})w(v_{\frac{k+j}{2}-1})<0$. Suppose $| w(v_{0})|-| w(v_{j})|=\alpha$. Then $w(v_{0})=-\mathrm{sgn}(w(v_{j}))(| w(v_{j})|+\alpha)$. Let $F=(f(v_{0})$, $f(v_{1})$, $\ldots$, $f(v_{n-1}))^{T}$ be a vector satisfying that
$$\left \{\begin{array}{ll}
 f(v_{i})=w(v_{i}),\ & \ v_{i}\in V(\Theta(j,k))\setminus \{v_{0}\};\\
\\ f(v_{0})=-\mathrm{sgn}(w(v_{\frac{k+j}{2}-1}))(| w(v_{\frac{k+j}{2}-1})|+\alpha).\ & \  \end{array}\right.$$
Then $|f(v_{0})|> |w(v_{0})|$, $F^{T}Q(G^{'})F=W^{T}Q(\Theta(j,k))W$, $F^{T}F>W^{T}W$, and then it follows that $q(G^{'})\leq \frac{F^{T}Q(G^{'})F}{F^{T}F}<\frac{W^{T}Q(\Theta(j,k))W}{W^{T}W}=q(\Theta(j,k))$. Note that $G^{'}\cong \Theta$. Then $q(\Theta)< q(\Theta(j,k))$.

Consequently, this lemma follows from the above discussion.
 \ \ \ \ \ $\Box$
\end{proof}

\begin{lemma}\label{le03.09.01} 
Any eigenvector of $\Theta$ corresponding to $q(\Theta)$ does not satisfy the conclusion in Lemma \ref{le03.06}.

\end{lemma}

\begin{proof}
We prove this lemma by contradiction. Suppose there is an eigenvector $Y=(y(v_{0})$, $y(v_{1})$, $\ldots$, $y(v_{n-1}))^{T}\in R^{n}$ of $\Theta$ corresponding to $q(\Theta)$ satisfying the conclusion in Lemma \ref{le03.06}. Then by Lemma \ref{le03.06}, it is known that (1) $y(v_{0})=y(v_{1})=0$; (2) $y(v_{2})\neq 0$; (3) $y(v_{i})=-y(v_{n+1-i})$ for $2\leq i\leq\frac{n}{2}$.

Let $G_{1}=\Theta+v_{0}v_{1}$. Note that $Y^{T}Q(G_{1})Y=Y^{T}Q(\Theta)Y$. By Lemma \ref{le02,03}, it follows that $q(\Theta)\leq q(G_{1})\leq \frac{Y^{T}Q(G_{1})Y}{Y^{T}Y}=\frac{Y^{T}Q(\Theta)Y}{Y^{T}Y}=q(\Theta)$. Thus $q(\Theta)= q(G_{1})$ and $Y$ is also an eigenvector of $G_{1}$ corresponding to $q(G_{1})$. Now, we let $G_{2}=G_{1}-v_{0}v_{n-1}$. Using Lemma \ref{le02,03} again gets that $q(G_{2})\leq q(G_{1})$. Note $G_{2}=\Theta(1,2)$ now. Using Lemma \ref{le03.08} and Lemma \ref{le03.09} gets that $q(\Theta)<q(G_{2})$. Then we get a contradiction that $q(\Theta)<q(G_{2})\leq q(G_{1})=q(\Theta)$. This makes the lemma hold.
 \ \ \ \ \ $\Box$
\end{proof}

\begin{lemma}\label{le03.010} 
Let $G\in \mathcal {H}$, $V(G) = \{v_{0}, v_{1}, \ldots, v_{n-1}\}$, $X=(x(v_{0})$, $x(v_{1})$, $\ldots$, $x(v_{n-1}))^{T}\in R^{n}$ be
an eigenvector of $G$ corresponding to $q(G)$,
and $|x(v_{\mu})|=\max\{|x(v_{i})|\,|\, v_{i}\in V(G)\}$, $v_{\mu}$ be in an odd cycle $C$. If $L(C)= n-1$, then $q(\Theta)\leq q(G) $ with equality if and only if $G\cong \Theta$.
\end{lemma}

\begin{proof}
Suppose that $G\ncong \Theta$, $C=v_{1}v_{2}\cdots v_{n-1}v_{1}$. Then
$V(G)\setminus V(C)=\{v_{0}\}$. Note that $G$ is 2-connected. By Lemma \ref{cl03.02}, there are two edges $v_{0}v_{j}$ and $v_{0}v_{k}$ where $1\leq j<k\leq n-1$. Then $C+v_{0}v_{j}+v_{0}v_{k}$ is a $\Theta(j,k)$. Thus we let $\Theta(j,k)=C+v_{0}v_{j}+v_{0}v_{k}$ here, and denote by $P_{1}(\Theta(j,k))=v_{j}v_{j+1}\cdots v_{k-1}v_{k}$, $P_{2}(\Theta(j,k))=C-\{v_{j+1}$, $\ldots$, $v_{k-1}\}=v_{j}v_{\eta_{1}}v_{\eta_{2}}\cdots v_{\eta_{z}}v_{k}$ if $k-j< n-2$; $P_{2}(\Theta(j,k))=v_{j}v_{k}$ if $k-j= n-2$. By Lemma \ref{le02,03}, we know that $q(\Theta(j,k))\leq q(G)$.
Note that $C$ is an odd cycle. Then one of $\|V(P_{1}(\Theta(j,k)))\|$ and $\|V(P_{2}(\Theta(j,k)))\|$ is odd. Without loss of generality, assume that $\|V(P_{1}(\Theta(j,k)))\|$ is odd. Then $k+j$ is even.

{\bf Case 1} $\frac{k+j}{2}-1>j$. Then $\Theta(j,k)\ncong \Theta$. Consequently, there are two cases to consider. One case is that there is an eigenvector $Y$ satisfying the conclusion in Lemma \ref{le03.06}. The other one is that any eigenvector of $\Theta(j,k)$ corresponding to $q(\Theta(j,k))$ does not satisfy the conclusion in Lemma \ref{le03.06}.
For these two cases, $q(\Theta)< q(\Theta(j,k))$ follows from Lemma \ref{le03.08} and Lemma \ref{le03.09} respectively. Thus $q(\Theta)< q(G)$.

{\bf Case 2} $\frac{k+j}{2}-1=j$. Note that $\Theta(j,k)\cong \Theta$, $G\neq \Theta$. It follows that $G\neq \Theta(j,k)$. Then there is at least one edge in $G$ which is not in $\Theta(j,k)$. Suppose that $v_{\alpha}v_{\tau}\in E(G)$ but $v_{\alpha}v_{\tau}\not\in E(\Theta(j,k))$. Assume that $q(\Theta(j,k))= q(G)$ and $W=(w(v_{0})$, $w(v_{1})$, $\ldots$, $w(v_{n-1}))^{T}\in R^{n}$ is an eigenvector of $G$ corresponding to $q(G)$. Note that $q(\Theta(j,k))\leq \frac{W^{T}Q(\Theta(j,k))W}{W^{T}W}\leq \frac{W^{T}Q(G)W}{W^{T}W}=q(G)$. Thus it follows that $W^{T}Q(\Theta(j,k))W=W^{T}Q(G)W$, $W=(w(v_{0})$, $w(v_{1})$, $\ldots$, $w(v_{n-1}))^{T}\in R^{n}$ is also an eigenvector of $\Theta(j,k)$ corresponding to $q(\Theta(j,k))$, and follows that $(w(v_{\alpha})+w(v_{\tau}))^{2}=0$ further. Note that $\Theta(j,k)\cong \Theta$. By Lemma \ref{le03.07} and Lemma \ref{le03.09.01}, it follows that $w(v_{\alpha})+w(v_{\tau})\neq 0$ which contradicts that $(w(v_{\alpha})+w(v_{\tau}))^{2}=0$. As a result, $q(\Theta(j,k))< q(G)$ follows, and furthermore, $q(\Theta)< q(G)$ follows.

Then the lemma follows from the above discussion. \ \ \ \ \ $\Box$
\end{proof}

\begin{theorem}\label{le03.011} 
Let $G\in \mathcal {H}$. Then $q(\Theta)\leq q(G) $ with equality if and only if $G\cong \Theta$.
\end{theorem}

\begin{proof}
This result follows from Lemma \ref{le03.05} and Lemma \ref{le03.010}.
\ \ \ \ \ $\Box$
\end{proof}

{\bf Proof of Theorem \ref{th01.01}}
This theorem follows from Theorem \ref{le03.04} and Theorem \ref{le03.011}. \ \ \ \ \ $\Box$

Theorem \ref{th01.02} follows from Theorem \ref{th01.01}  as a corollary directly.

\small {

}

\end{document}